\documentclass{amsart}

\usepackage[active]{srcltx}

\usepackage{amsmath,amssymb,amsthm,amscd}
\newtheorem{theorem}{Theorem}[section]
\newtheorem{lemma}[theorem]{Lemma}
\newtheorem{proposition}[theorem]{Proposition}
\newtheorem{corollary}[theorem]{Corollary}

\newtheorem{remark}[theorem]{Remark}
\newtheorem{example}[theorem]{Example}

\begin{document}

\title[finitely generated powers]{Finitely generated powers of prime ideals}
\author{Fran\c{c}ois Couchot}

\address{Universit\'e de Caen Normandie, CNRS UMR
  6139 LMNO,
F-14032 Caen, France}
\email{francois.couchot@unicaen.fr}  


\begin{abstract}
Let $R$ be a commutative ring. If $P$ is a maximal ideal of $R$ with a finitely generated power then we prove that $P$ is finitely generated if $R$ is either locally coherent or arithmetical or a polynomial ring over a ring of global dimension $\leq 2$. And, if $P$ is a prime ideal of $R$ with a finitely generated power then we show that $P$ is finitely generated if $R$ is either a reduced coherent ring or a polynomial ring over  a reduced arithmetical ring. These results extend a theorem of Roitman, published in 2001, on prime ideals of coherent integral domains.
\end{abstract}

\maketitle

\section{Introduction}
All rings are commutative and unitary. In this paper the following question is studied:

{\bf question A}: Suppose that some power $P^n$ of the prime ideal $P$ of a ring $R$ is finitely generated. Does it follow that $P$ is finitely generated?

When $P$ is maximal it is the {\it question 0.1} of \cite{GiHeRo99}, a paper by Gilmer, Heinzer and Roitman. The first author posed this question in \cite[p.74]{Gil72}. In \cite{GiHeRo99} some positive answers are given to the {\it question 0.1} (see \cite[for instance, Theorem 1.24]{GiHeRo99}), but also some negative answers (see \cite[Example 3.2]{GiHeRo99}). The authors proved a very interesting result ( \cite[Theorem 1.17]{GiHeRo99}): a reduced ring $R$ is Noetherian if each of its prime ideals has a finitely generated power. This {\it question 0.1} was recently studied in \cite{MaZe15} by Mahdou and Zennayi, where some examples of rings with positive answers are given, but also some examples with negative responses. In \cite{Roi01} Roitman investigated the {\bf question A}. In particular, he proved that $P$ is finitely generated if $R$ is a coherent integral domain (\cite[Theorem 1.8]{Roi01}).

We first study {\it question 0.1} in Section \ref{S:max}. It is proven that $P$ is finitely generated if $R$ is either locally coherent or arithmetical. In Section \ref{S:prime} we investigate {\bf question A} and extend the Roitman's result. We get a positive answer when $R$ is a reduced ring which is either coherent or arithmetical. If $R$ is not reduced, we obtain a positive answer for all prime ideals $P$, except if $P$ is minimal and not maximal. In Section \ref{S:pf}, by using Greenberg and Vasconcelos's results, we deduce that {\bf question A} has also a positive response if $R$ is a polynomial ring over either a reduced arithmetical ring or a ring of global dimension $\leq 2$. In Section \ref{S:min}, we consider rings of constant functions defined over a totally disconnected compact space $X$ with values in a ring $O$ for which a quotient space of $\mathrm{Spec}\ O$ has a unique point, and we examine when these rings give a positive answer to our questions. This allows us to provide some examples and counterexamples. 

 We denote respectively $\mathrm{Spec}\ R$, $\mathrm{Max}\ R$ and $\mathrm{Min}\ R,$ the
space of prime ideals, maximal ideals and minimal prime ideals of
$R$, with the Zariski topology. If $A$ is a subset of $R$, then we denote $(0:A)$ its annihilator and 
\[V(A) = \{ P\in\mathrm{Spec}\ R\mid A\subseteq P\}\ \mathrm{and}\ \ D(A) =\mathrm{Spec}\ R\setminus V(A).\]

\section{Powers of maximal ideals}\label{S:max}
 
Recall that a ring $R$ is {\bf coherent} if each finitely generated ideal is finitely presented. It is well known that $R$ is coherent if and only if $(0:r)$ and $A\cap B$ are finitely generated for each $r\in R$ and any two finitely generated ideals $A$ and $B$.

\begin{theorem}\label{T:cohmax}
Let $R$ be a  coherent ring. If $P$ is a maximal ideal such that $P^n$ is finitely generated for some integer $n>0$ then $P$ is finitely generated too.
\end{theorem}
\begin{proof}
First, suppose there exists an integer $n>0$ such that $P^n=0$. So, $R$ is local of maximal ideal $P$. We can choose $n$ minimal. If $n=1$ then $P$ is clearly finitely generated. Suppose $n>1$. It follows that $P^{n-1}\ne 0$. So, $P=(0:r)$ for each $0\ne r\in P^{n-1}$. Since $R$ is coherent, $P$ is finitely generated. Now, suppose that $P^n$ is finitely generated for some integer $n\geq 1$. If $R'=R/P^n$ and $P'=P/P^n$ then $R'$ is coherent and $P'^n=0$. From above we  deduce that $P'$ is finitely generated. Hence $P$ is finitely generated too.
\end{proof} 

The following theorem can be proven by using \cite[Lemma 1.8]{GiHeRo99}.
\begin{theorem}
\label{T:main} Let $R$ be a ring. Suppose that $R_L$ is coherent for each maximal ideal $L$. If $P$ is a maximal ideal such that $P^n$ is finitely generated for some integer $n>0$ then $P$ is finitely generated too.
\end{theorem}
\begin{proof}
Suppose that $P^n$ is generated by $\{x_1,\dots,x_k\}$. Let $L\ne P$ be a maximal ideal. Let $s\in P\setminus L$. Then $s^n\in P^n\setminus L$. It follows that $s^nR_L=P^nR_L=PR_L=R_L$. So, there exists $i,\ 1\leq i\leq k$ such that $PR_L=x_iR_L$. Since $R_P$ is coherent, $PR_P$ is finitely generated by Theorem \ref{T:cohmax}. So, there exist $y_1, \dots, y_m$ in $P$ such that $PR_P=y_1R_P+\dots +y_mR_P$. Let $Q$ be the ideal generated by $\{x_1,\dots,x_k\}\cup\{y_1,\dots,y_m\}$. Then $Q\subseteq P$ and it is easy to check that $QR_L=PR_L$ for each maximal ideal $L$. Hence $P=Q$ and $P$ is finitely generated. 
\end{proof}
 
A ring $R$ is a {\bf chain ring} if its lattice of ideals is totally ordered by inclusion, and $R$ is {\bf arithmetical} if $R_P$ is a  chain ring for each maximal ideal $P$.

\begin{theorem}\label{T:arith}
Let $R$ be an arithmetical ring. If $P$ is a maximal ideal such that $P^n$ is finitely generated for some integer $n>0$ then $P$ is finitely generated too.
\end{theorem}
\begin{proof}
First, assume that $R$ is local. Let $P$ be its maximal ideal. Suppose that $P$ is not finitely generated and let $r\in P$. Since $P\ne Rr$ there exists $a\in P\setminus Rr$. So, $r=ab$ with $b\in P$. It follows that $P^2=P$ and $P^n=P$ for each integer $n>0$. So, $P^n$ is not finitely generated for each integer $n>0$. Now, we do as in the proof of Theorem \ref{T:main} to complete the demonstration. 
\end{proof}
\begin{remark}
\textnormal{There exist arithmetical rings which are not coherent. In \cite{MaZe15} several other examples of non-coherent rings which satisfy the conclusion of the previous theorem are given.}
\end{remark}

Let $R$ be a ring. For a polynomial $f\in R[X]$, denote  by $c(f)$ (the content of $f$) the ideal of $R$ generated by the coefficients of $f$. We say that $R$ is {\bf Gaussian} if $c(fg)=c(f)c(g)$ for any two polynomials $f$ and $g$ in $R[X]$ (see \cite{Tsa65}). A ring $R$ is said to be a {\bf fqp-ring} if each finitely generated ideal $I$ is projective over $R/(0:I)$ (see \cite[Definition 2.1 and Lemma 2.2]{AbJaKa11}).

By \cite[Theorem 2.3]{AbJaKa11} each arithmetical ring is a fqp-ring and each fqp-ring is Gaussian, but the converses do not hold. The following examples show that Theorem \ref{T:arith} cannot be extented to the class of fqp-rings and the one of Gaussian rings.

\begin{example}
Let $R$ be a local ring and $P$ its maximal ideal. Assume that $P^2=0$. Then it is easy to see that $R$ is a fqp-ring. But $P$ is possibly not finitely generated.
\end{example}

\begin{example}
Let $A$ be a valuation domain (a chain domain), $M$ its maximal ideal generated by $m$ and  $E$ a vector space over $A/M$. Let $R=\{\binom{a\ e}{0\ a}\mid a\in A,\ e\in E\}$ be the trivial ring extension of $A$ by $E$. By \cite[Corollary 2.2 and Theorem 4.2]{Cou15} $R$ is a local Gaussian ring which is not a fqp-ring. Let $P$ be its maximal ideal. Then $P^2$ is generated by $\binom{m^2\ 0}{0\ m^2}$. But, if $E$ is of infinite dimension over $A/M$ then $P$ is not finitely generated over $R$ (see also \cite[Theorem 2.3({\it iv}){\bf a})]{MaZe15}).   
\end{example}

\section{Powers of prime ideals}\label{S:prime}

By \cite[Theorem 1.8]{Roi01}, if $R$ is a coherent integral domain then each prime ideal with a finitely generated power  is finitely generated too. The following example shows that this result does not extend to any coherent ring.

\begin{example}\label{E:exa}
Let $D$ be a valuation domain. Suppose there exists a non-zero prime ideal $L'$ which is not maximal. Moreover assume that $L'\ne L'^2$ and let $d\in L'\setminus L'^2$. If $R=D/Dd$ and $L=L'/Dd$, then $R$ is a coherent ring, $L$ is not finitely generated and $L^2=0$.
\end{example}

\begin{remark}
Let $R$ be an arithmetical ring. In the previous example we use the fact that each non-zero prime ideal $L$ which is not maximal is not finitely generated. In Theorem \ref{T:arith1} we shall prove that $L^n$ is not finitely generated for each integer $n>0$ if $L$ is not minimal. 
\end{remark}

In the sequel let $\Phi=\mathrm{Max}\ R\cup(\mathrm{Spec}\ R\setminus\mathrm{Min}\ R)$ for any ring $R$.

The proof of the following theorem is similar to that of \cite[Theorem 1.8]{Roi01}.

\begin{theorem}\label{T:mainP}
Let $R$ be a coherent ring. Then, for any $P\in\Phi$, $P$ is finitely generated if $P^n$ is finitely generated for some integer $n>0$.
\end{theorem}
\begin{proof} Let $P\in\Phi$  such that $P^k$ is finitely generated for some integer $k>0$. By Theorem \ref{T:cohmax} we may assume that $P$ is not maximal. So, there exists a minimal prime ideal $P'$ such that $P'\subset P$. It follows that  $P^n\ne 0$ for each integer $n>0$. By \cite[Lemma 1.7]{Roi01} there exist an integer $n>1$ such that $P^n$ is finitely generated and $a\in P^{n-1}\setminus P^{(n)}$ where $P^{(n)}$ is the inverse image of $P^nR_P$ by the natural map $R\rightarrow R_P$. This implies that $aP=aR\cap P^n$. We may assume that $a\notin P'$, else, we replace $a$ with $a+b$ where $b\in P^n\setminus P'$. Since $R$ is coherent, $aP$ and $(0:a)$ are finitely generated. From $a\notin P'$ we deduce $(0:a)\subseteq P'\subset P$, whence $P\cap (0:a)=(0:a)$. Hence $P$ is finitely generated.
\end{proof}

\begin{corollary}
\label{C:mainCoh} Let $R$ be a reduced coherent ring. Then, for any prime ideal $P$, $P$ is finitely generated if $P^n$ is finitely generated for some integer $n>0$.
\end{corollary}
\begin{proof}
Let $P$ be a prime ideal of $R$ such that $P^n$ is finitely generated for some integer $n>1$. We may assume that $P\ne 0$ and  by Theorem \ref{T:mainP} that $P$ is minimal. So, $P^n\ne 0$. It is easy to check that $(0:P)=(0:P^n)$ because $R$ is reduced. Since $R$ is coherent, it follows that $(0:P)$ is finitely generated. On the other hand, since $P^n$ is finitely generated, there exists $t\in (0:P^n)\setminus P$. This implies that $P=(0:(0:P))$. We conclude that $P$ is finitely generated.
\end{proof}

An exact sequence of $R$-modules $0 \rightarrow F \rightarrow E \rightarrow G \rightarrow 0$  is {\bf pure}
if it remains exact when tensoring it with any $R$-module. Then, we say that $F$ is a \textbf{pure} submodule of $E$. The following proposition is well known.
\begin{proposition} \cite[Proposition 2.4]{Couc09} 
\label{P:purIdeal} Let $A$ be an ideal of a ring $R$. The following conditions are equivalent:
\begin{enumerate}
\item $A$ is a pure ideal of $R$;
\item for each finite family $(a_i)_{1\leq i\leq n}$ of elements of $A$ there exists $t\in A$ such that $a_i=a_it,\ \forall i,\ 1\leq i\leq n$;
\item for all $a\in A$ there exists $b\in A$ such that $a=ab$ (so, $A=A^2$);
\item $R/A$ is a flat $R$-module.
\end{enumerate} 
Moreover:
\begin{itemize}
\item if $A$ is finitely generated, then $A$ is pure if and only if it is generated by an idempotent;
\item if $A$ is pure, then $R/A=S^{-1}R$ where $S=1+A$.
\end{itemize} 
\end{proposition}

If $R$ is a ring, we consider on $\mathrm{Spec}\ R$ the equivalence relation $\mathcal{R}$ defined by   $L\mathcal{R} L'$ if there exists a finite sequence of prime ideals $(L_k)_{1\leq k\leq n}$ such that $L=L_1,$ $L'=L_n$ and $\forall k,\ 1\leq k\leq (n-1),$ either $L_k\subseteq L_{k+1}$ or $L_k\supseteq L_{k+1}$. We denote by $\mathrm{pSpec}\ R$ the quotient space of $\mathrm{Spec}\ R$ modulo $\mathcal{R}$ and by $\lambda: \mathrm{Spec}\ R\rightarrow\mathrm{pSpec}\ R$ the natural map. The quasi-compactness of $\mathrm{Spec}\ R$ implies the one of $\mathrm{pSpec}\ R$, but generally $\mathrm{pSpec}\ R$  is not  $T_1$: see \cite[Propositions 6.2 and 6.3]{Laz67}. 

\begin{lemma}\label{L:pure} \cite[Lemma 2.5]{Couc09}. Let $R$ be a ring and let $C$ a closed subset of $\mathrm{Spec}\ R$. Then $C$ is the inverse image of a closed subset of $\mathrm{pSpec}\ R$ by $\lambda$ if and only if $C=V(A)$ where $A$ is a pure ideal. Moreover, in this case, $A=\cap_{P\in C}\ker(R\rightarrow R_P)$.
\end{lemma}

In the sequel, for each $x\in\mathrm{pSpec}\ R$ we denote by $A(x)$ the unique pure ideal which verifies   $\overline{\{x\}}=\lambda(V(A(x)))$, where $\overline{\{x\}}$ is the closure of $\{x\}$ in $\mathrm{pSpec}\ R$. 

\begin{theorem}\label{T:main1} Let $R$ be a ring. Assume that $R/A(x)$ is coherent for each $x\in\mathrm{pSpec}\ R$. Then, for any $P\in\Phi$, $P$ is finitely generated if $P^n$ is finitely generated for some integer $n>0$. 
\end{theorem}
\begin{proof}
Let $P\in\Phi$ and $I=A(\lambda(P))$. Suppose that $P^n$ is generated by $\{x_1,\dots,x_k\}$. Let $L$ be a maximal ideal such that $I\nsubseteq L$. As in the proof of Theorem \ref{T:main} we show that $PR_L=x_iR_L$ for some integer $i,\ 1\leq i\leq k$. By Theorem \ref{T:mainP} $P/I$ is finitely generated over $R/I$. So, there exist $y_1, \dots, y_m$ in $P$ such that ($y_1+I,\dots ,y_m+I$) generate $P/I$. Let $Q$ be the ideal generated by $\{x_1,\dots,x_k\}\cup\{y_1,\dots,y_m\}$. Then $Q\subseteq P$ and it is easy to check that $QR_L=PR_L$ for each maximal ideal $L$. Hence $P=Q$ and $P$ is finitely generated. 
\end{proof}

From Corollary \ref{C:mainCoh} and Theorem \ref{T:main1} we deduce the following.

\begin{corollary}\label{C:locCoh}
Let $R$ be a reduced ring. Assume that $R/A(x)$ is coherent for each $x\in\mathrm{pSpec}\ R$. Then, for any prime ideal $P$, $P$ is finitely generated if $P^n$ is finitely generated for some integer $n>0$. 
\end{corollary}

\begin{theorem}\label{T:arith1}
Let $R$ be an arithmetical ring. Then, for any $P\in\Phi$, $P$ is finitely generated if $P^n$ is finitely generated for some integer $n>0$.
\end{theorem}
\begin{proof}
Let $P$ be a prime ideal. By Theorem \ref{T:arith} we may assume that $P$ is not maximal.  Let $M$ be a maximal ideal containing $P$. If $P$ is not minimal then $P^nR_M$ contains strictly the minimal prime ideal of $R_M$ for each integer $n>0$. So, $P^nR_M\ne 0$ for each integer $n>0$. On the other hand, since $R_M$ is a chain ring it is easy to check that $PR_M=MPR_M$. It follows that $P^nR_M=MP^nR_M$ for each integer $n>0$. By Nakayama Lemma we deduce that $P^nR_M$ is not finitely generated over $R_M$. Hence, $P^n$ is not finitely generated for each integer $n>0$.
\end{proof}

\begin{remark}
\textnormal{Example \ref{E:exa} shows that the assumption "$P\in\Phi$" cannot be omitted in some previous results. However, if each minimal prime ideal which is not maximal is idempotent then the conclusions hold for each prime ideal $P$.}
\end{remark}

\begin{proposition}\label{P:min1} Let $R$ be a ring. Let $P$ be a minimal prime ideal such that $P^n$ is finitely generated for some integer $n>0$. Then $P$ is an isolated point of $\mathrm{Min}\ R$.
\end{proposition}
\begin{proof}
Let $N$ be the nilradical of $R$. For any finitely generated ideal $I$ we easily check that $V(I)\cap\mathrm{Min}\ R=D((N:I))\cap\mathrm{Min}\ R$. Hence it is a clopen (closed and open) subset of $\mathrm{Min}\ R$. Since $V(P^n)\cap\mathrm{Min}\ R=\{P\}$, $P$ is an isolated point of $\mathrm{Min}\ R$ if $P^n$ is finitely generated.
\end{proof}

From Theorems \ref{T:main1} and \ref{T:arith1} and Proposition \ref{P:min1} we deduce the following corollary.

\begin{corollary}
Let $R$ be a ring. Assume that $\mathrm{Min}\ R$ contains no isolated point and $R$ satisfies one of the following conditions:
\begin{itemize}
\item $R/A(x)$ is coherent for each $x\in\mathrm{pSpec}\ R$;
\item $R$ is arithmetical.
\end{itemize}
Then, each prime ideal with a finitely generated power is  finitely generated too. 
\end{corollary}

\begin{proposition}\label{P:min} Let $R$ be a ring for which each prime ideal contains only one minimal prime ideal. Let $P$ be a minimal prime ideal such that $P^n$ is finitely generated for some integer $n>0$. Then $\lambda(P)$ is an isolated point of $\mathrm{pSpec}\ R$.
\end{proposition}
\begin{proof}
Let $P$ be a minimal prime ideal and $A=A(\lambda(P))$. Clearly $\lambda(P)=V(P)=V(A)$. We have  $A^2=A$. From $A\subseteq P$ we deduce that $A\subseteq P^2$. It follows that $A\subseteq P^n$ for each integer $n>0$. Suppose that $P^n$ is finitely generated for some integer $n>0$. Since $P/A$ is the nilradical of $R/A$, $P^m=A$ for some integer $m\geq n$. We deduce that $P^m=Re$ for some idempotent $e$ of $R$ by Proposition \ref{P:purIdeal}. It follows that $\lambda(P)=V(P^m)=D(1-e)$. Hence $\lambda(P)$ is an isolated point of $\mathrm{pSpec}\ R$. 
\end{proof}

\section{pf-rings}\label{S:pf}

Now, we consider the rings $R$ for which each prime ideal contains a unique minimal prime ideal. So, the restriction $\lambda'$ of $\lambda$ to $\mathrm{Min}\ R$ is bijective. In this case, for each minimal prime ideal $L$ we put $A(L)=A(\lambda(L))$. By \cite[Proposition IV.1]{Cou07} $\mathrm{pSpec}\ R$ is Hausdorff and  $\lambda'$ is a homeomorphism if and only if $\mathrm{Min}\ R$ is compact. We deduce the following from Lemma \ref{L:pure}.

\begin{proposition}
\label{P:pireg}
Let $R$ be a ring. Assume that each prime ideal contains a unique minimal prime ideal. Then, for each minimal prime ideal $L$, $V(L)=V(A(L))$. Moreover, if $R$ is reduced then $A(L)=L$.
\end{proposition}
\begin{proof} If $R$ is reduced, then, for each $P\in V(L)$, $LR_P=0$, whence $L=\ker(R\rightarrow R_P)$.
\end{proof}

As in \cite[p.14]{Vas76} we say that a ring $R$ is a {\bf pf-ring} if one of the following equivalent conditions holds:
\begin{enumerate}
\item $R_P$ is an integral domain for each maximal ideal $P$;
\item each principal ideal of $R$ is flat;
\item each cyclic submodule of a flat $R$-module is flat.
\end{enumerate}
Moreover, if $R$ is a pf-ring then each prime ideal $P$ contains a unique minimal prime ideal $P'$ and $A(P')=P'$ by Proposition \ref{P:pireg}.

So, from the previous section and the fact that each minimal prime ideal of a pf-ring is idempotent, we deduce the following three results. Let us observe that each prime ideal of an arithmetical ring $R$ contains a unique minimal prime ideal because $R_P$ is a chain ring for each maximal ideal $P$.

\begin{corollary}
Let $R$ be a coherent pf-ring. Then each prime ideal with a finitely generated power is finitely generated too. 
\end{corollary}

\begin{corollary}
\label{C:mainbis} Let $R$ be a pf-ring. Assume that $R/L$ is coherent for each minimal prime ideal $L$. Then each prime ideal with a finitely generated power is finitely generated too. 
\end{corollary}

\begin{corollary}
Let $R$ be a reduced arithmetical ring. Then each prime ideal with a finitely generated power is finitely generated too.
\end{corollary}

The following three corollaries allows us to give some examples of pf-ring satisfying the conclusion of Corollary \ref{C:mainbis}. Let $n$ be an integer $\geq 0$ and $G$ a module over a ring $R$. We say that pd$\ G\leq n$ if Ext$^{n+1}_R(G,H)=0$ for each $R$-module $H$.

\begin{corollary}
Let $R$ be a coherent ring. Assume that each finitely generated ideal $I$ satisfies pd$\ I<\infty$. Then each prime ideal with a finitely generated power is finitely generated too.
\end{corollary}
\begin{proof}
By, either \cite[Th\'eor\`eme A]{Ber71} or \cite[Corollary 6.2.4]{Gla89}, $R_P$ is an integral domain for each maximal ideal $P$. So, $R$ is a pf-ring.
\end{proof}

\begin{corollary}
Let $A$ be a ring and $X=\{X_{\lambda}\}_{\lambda\in\Lambda}$ a set of indeterminates. Consider the polynomial ring $R=A[X]$. Assume that $A$ is reduced and arithmetical. Then each prime ideal of $R$ with a finitely generated power is finitely generated too.
\end{corollary}
\begin{proof} Let $P$ be a maximal ideal of $R$ and $P'=P\cap A$. Thus $R_P$ is a localization of $A_{P'}[X]$. Since $A_{P'}$ is a valuation domain, $R_P$ is an integral domain. So, $R$ is a pf-ring. Now, let $P$ be a minimal prime ideal of $R$ and $L$ be a minimal prime ideal of $A$ contained in $P\cap A$. We put $A'=A/L$ and $R'=A'[X]$. So, $A'$ is an arithmetical domain (a Pr\"ufer domain). By \cite[3.(b)]{GrVa76} $R'$ is coherent. Since $R/P$ is flat over $R$ and $R'$, $R/P$ is a localization of $R'$. Hence $R/P$ is coherent. We conclude by Corollary \ref{C:mainbis}.
\end{proof}
 
Let $n$ be an integer $\geq 0$. We say that a ring $R$ is of global dimension $\leq n$ if pd$\ G\leq n$ for each $R$-module $G$.

\begin{corollary}
Let $A$ be a ring and $X=\{X_{\lambda}\}_{\lambda\in\Lambda}$ a set of indeterminates. Consider the polynomial ring $R=A[X]$. Assume that $A$ is of global dimension $\leq 2$. Then each prime ideal of $R$ with a finitely generated power is finitely generated too.
\end{corollary}
\begin{proof}
Let $P$ be a maximal ideal of $R$ and $P'=P\cap A$. Thus $R_P$ is a localization of $A_{P'}[X]$. Since $A_{P'}$ is an integral domain by \cite[Lemme 2]{LeB71}, $R_P$ is an integral domain. So, $A$ and $R$ are pf-rings. By \cite[Proposition 2]{LeB71} $A/L$ is coherent for each minimal prime ideal $L$. Now, we conclude as in the proof of the previous corollary, by using  \cite[(4.4) Corollary ]{GrVa76}.
\end{proof}

\section{Rings of locally constant functions}
\label{S:min}

 A topological space is called \textbf{totally disconnected} if each of its connected components contains only one point. Every Hausdorff topological space $X$ with a base of clopen (closed and open) neighbourhoods is totally disconnected and the converse holds if $X$ is compact (see \cite[Lemma 29.6]{Wil70}).

\begin{proposition}\label{P:disc}
Let $X$ be a totally disconnected compact space, let $O$ be a ring with a unique point in $\mathrm{pSpec}\ O$. Let $R$ be the ring of all locally constant maps from $X$ into $O$. Then, $\mathrm{pSpec}\ R$ is homeomorphic to $X$ and $R/A(z)\cong O$ for each $z\in\mathrm{pSpec}\ R$.
\end{proposition}
\begin{proof} If $U$ is a clopen subset of $X$ then there exists an idempotent   $e_U$  defined by $e_U(x)=1$ if $x\in U$ and $e_U(x)=0$ else. Let $x\in X$ and $\phi_x:R\rightarrow O$ be the map defined by $\phi_x(r)=r(x)$ for every $r\in R$. Clearly $\phi_x$ is a ring homomorphism, and since $R$ contains all the constant maps, $\phi_x$ is surjective. Let $x\in X,\ r\in\ker (\phi_x)$ and $U=\{y\in X\mid r(y)\ne 0\}$. Then $U$ is a clopen subset. It is easy to check that $e_U\in\ker (\phi_x)$ and $r=re_U$. Since $\ker (\phi_x)$ is generated by idempotents, $R/\ker (\phi_x)$ is flat over $R$. For each $x\in X$, let $\Pi(x)$ be the image of $\mathrm{Spec}\ O$ by $\lambda\circ\phi_x^a$ where $\phi_x^a:\mathrm{Spec}\ O\rightarrow\mathrm{Spec}\ R$ is the continuous map induced by $\phi_x$.  We shall prove that $\Pi:X\rightarrow\mathrm{pSpec}\ R$ is a homeomorphism. Clearly, $V(\ker(\phi_x))\subseteq\Pi(x)$. Conversely, let $P\in\Pi(x)$. Then there exists $L\in V(\ker(\phi_x))$ such that $P\mathcal{R} L$. We may assume that $L\subseteq P$ or $P\subseteq L$. The first case is obvious. For the second case let $e$ an idempotent of $\ker(\phi_x)$. Then, $e\in L$, $(1-e)\notin L$, $(1-e)\notin P$ and $e\in P$. We conclude that $V(\ker(\phi_x))=\Pi(x)$ because $\ker(\phi_x)$ is generated by its idempotents. Let $x,y\in X$, $x\ne y$. By using the fact there exists a clopen subset $U$ of $X$ such that $x\in U$ and $y\notin U$ then $e_U\in\ker(\phi_y)$ and $(1-e_U)\in\ker(\phi_x)$. So, $\ker(\phi_x)+\ker(\phi_y)=R$, whence $\Pi$ is injective. By way of contradiction suppose there exists a prime ideal $P$ of $R$ such that $\ker(\phi_x)\nsubseteq P$ for each $x\in X$. There exists an idempotent $e_x'\in\ker(\phi_x)\setminus P$ whence $e_x=(1-e_x')\in P\setminus\ker(\phi_x)$. Let $V_x$ be the clopen subset associated with $e_x$. Clearly $X=\cup_{x\in X}V_x$. Since $X$ is compact, a finite subfamily $(V_{x_i})_{1\leq i\leq n}$ covers $X$. We put $U_1=W_1=V_{x_1}$, and for $k=2,\dots,n$, $W_k=\cup_{i=1}^k V_{x_i}$ and $U_k=W_k\setminus W_{k-1}$. Then $U_k$ is clopen for each $k=1,\dots,n$. For $i=1,\dots,n$ let $\epsilon_i\in R$ be the idempotent associated with $U_i$. Since $U_i\subseteq V_{x_i}$, we have $\epsilon_i=e_{x_i}\epsilon_i$. So, $\epsilon_i\in P$ for $i=1,\dots,n$. It is easy to see that $1=\Sigma_{i=1}^n\epsilon_i$. We get $1\in P$. This is false. Hence $\Pi$ is bijective. We easily check that $x\in U$, where $U$ is a clopen subset of $X$, if and only if $\Pi(x)\subseteq D(e_U)$. Since $A(\Pi(x))=\ker(\phi_x)$ is generated by its idempotents, $\mathrm{pSpec}\ R$ has a base of clopen neighbourhoods. We conclude that $\Pi$ is a homeomorphism. 
\end{proof}

From Corollary \ref{C:locCoh} we deduce the following proposition.

\begin{proposition}
Let $R$ be the ring defined in Proposition \ref{P:disc}. Assume that $O$ is a reduced coherent ring.  Then, for any prime ideal $P$, $P$ is finitely generated if $P^n$ is finitely generated for some integer $n>0$.
\end{proposition}

\begin{proposition}\label{P:disc1} Let $R$ be the ring defined in Proposition \ref{P:disc}. Assume that $O$ has a unique minimal prime ideal $M$. Then, every prime ideal of $R$ contains only one minimal prime ideal and $\mathrm{Min}\ R$ is compact. If $M=0$ then $R$ is a pp-ring, i.e. each principal ideal is projective.
\end{proposition}
\begin{proof} If $P$ is a prime ideal of $R$ then there exists a unique $x\in X$ such that $P\in\Pi(x)$. So, $\phi_x^a(M)$ is the only minimal prime ideal contained in $P$.

Assume that $M=0$. Let $r\in R$, $e=e_U$ where $U$ is the clopen subset of $X$ defined by $U=\{x\in X\mid r(x)\ne 0\}$. We easily check that the map $Re\rightarrow Rr$ induced by the multiplication by $r$ is an isomorphism. This proves that $R$ is a pp-ring.

Let $R'$ be the ring obtained like $R$ by replacing $O$ with $O/M$. It is easy to see that $R'\cong R/N$ where $N$ is the nilradical of $R$. So, $\mathrm{Min}\ R$ and  $\mathrm{Min}\ R'$ are homeomorphic. Since $R'$ is a pp-ring, $\mathrm{Min}\ R$ is compact by \cite[Proposition 1.13]{Vas76}.
\end{proof}

From Theorems \ref{T:main1} and \ref{T:arith1} and Propositions \ref{P:min} and \ref{P:disc1} we deduce the following corollary.

\begin{corollary}\label{C:disc1} Let $R$ be the ring defined in Proposition \ref{P:disc}. Suppose that $O$ has a unique minimal prime ideal $M$. Assume that $O$ is either coherent or arithmetical and that one of the following conditions holds:
\begin{enumerate}
\item $M$ is either idempotent or finitely generated;
\item $X$ contains no isolated point.
\end{enumerate}
 Then, for any prime ideal $P$, $P$ is finitely generated if $P^n$ is finitely generated for some integer $n>0$.
\end{corollary}

\begin{example}
Let $R$ be the ring defined in Proposition \ref{P:disc}. Assume that:
\begin{itemize}
\item  $O$ is either coherent or arithmetical, with a unique minimal prime ideal $M$;
\item $M$ is not finitely generated and  $M^k=0$ for some integer $k>1$ (for example, $O$ is the ring $R$ defined in Example \ref{E:exa});
\item $X$ contains no isolated points (for example the Cantor set, see \cite[Section 30]{Wil70}).
\end{itemize}
Then the property "for each prime ideal $P$, $P^n$ is finitely generated for some integer $n>0$ implies $P$ is finitely generated" is satisfied by $R$, but not by $R/A(L)$ for each minimal prime ideal $L$.
\end{example}

From Theorems \ref{T:main} and \ref{T:arith} and Proposition \ref{P:min} we deduce the following corollary.

\begin{corollary}\label{C:disc} Let $R$ be the ring defined in Proposition \ref{P:disc}. Assume that $O$ is local with maximal ideal $M$. Then each prime ideal of $R$ is contained in a unique maximal ideal, and for each maximal ideal $P$, $R_P\cong O$. Moreover, if one of the following conditions holds:
\begin{enumerate}
\item $O$ is coherent;
\item $O$ is a chain ring;
\item $X$ contains no isolated point and $M$ is the sole prime ideal of $O$.
\end{enumerate}
then, for each maximal ideal $P$, $P^n$ finitely generated for some integer $n>0$ implies $P$ is finitely generated.
\end{corollary}

\begin{example}\label{E:count}
Let $R$ be the ring defined in Proposition \ref{P:disc}. Assume that $M$ is the sole prime ideal of $O$, $M$ is not finitely generated, $M^k=0$ for some integer $k>1$ and $X$ contains no isolated points. Then the property "for each maximal ideal $P$, $P^n$ is finitely generated for some integer $n>0$ implies $P$ is finitely generated" is satisfied by $R$, but not by $R_L$ for each maximal ideal $L$.
\end{example}

\section*{Acknowledgements}

This work was presented at the "Conference on Rings and Polynomials" held in Graz, Austria, July 3-8, 2016. I thank again the organizers of this conference.



\end{document}